\documentclass[twoside,11pt]{article}
\usepackage{pgm}
\pdfoutput=1

\usepackage[utf8]{inputenc}
\inputencoding{utf8}

\usepackage{hyperref,color,soul,booktabs} %
\usepackage{mathtools} %
\usepackage[ruled,lined]{algorithm2e} %
\usepackage{stackengine} %
\usepackage{paralist}
\usepackage{tikz} %
\usetikzlibrary{arrows.meta,calc,decorations.pathmorphing}
\setulcolor{blue}
\newcommand{\BibTeX}{\textsc{B\kern-0.1emi\kern-0.017emb}\kern-0.15em\TeX}

\renewcommand{\eqref}[1]{(\ref{eq:#1})}
\DeclarePairedDelimiterX\set[1]\lbrace\rbrace{#1}

\newcommand{\myempty}{\emptyset} 

\newcommand{\R}{\mathbb{R}} 

\newcommand{\trans}{^T}

\newcommand{\pa}{\mathrm{pa}}

\newcommand{\htr}{\mathrm{htr}}
\newcommand{\PD}{\mathrm{PD}}
\newcommand{\bA}{\mathbf{A}}
\newcommand{\bb}{\mathbf{b}}
\newcommand{\bc}{\mathbf{c}}

\newcommand{\cM}{\ensuremath{\mathcal M}}

\newcommand{\cY}{\ensuremath{\mathcal Y}}

\SetArgSty{} 

\ShortHeadings{Graphical Representations for Algebraic Constraints}{Van Ommen and Drton}

\begin{document}

\title{Graphical Representations for Algebraic Constraints\\ of Linear Structural Equations Models}

\author{\Name{Thijs van Ommen} \Email{t.vanommen@uu.nl}\\
   \addr Information and Computing Sciences, Utrecht University, The Netherlands\and
   \Name{Mathias Drton} \Email{mathias.drton@tum.de}\\
   \addr Mathematics, Technical University of Munich, Germany}

\maketitle

\begin{abstract}%
  The observational characteristics of a linear structural equation model can be effectively described by polynomial constraints on the observed covariance matrix. However, these polynomials can be exponentially large, making them impractical for many purposes. In this paper, we present a graphical notation for many of these polynomial constraints. The expressive power of this notation is investigated both theoretically and empirically.  %
\end{abstract}

\section{Introduction}

It is well known that for any directed acyclic graph, the set of observational distributions realizable by that graph can be described in terms of a set of conditional independence constraints that are imposed on the distributions. Our understanding of these constraints has enabled many sophisticated algorithms for learning the structure of Bayesian networks. %

When latent confounders are considered, as is important for purposes of causal discovery, the situation changes and additional constraints such as the Verma constraint \citep{Robins1986_VermaConstraint,VermaPearl1991_VermaConstraint} are needed in addition to conditional independences. Here we need to distinguish between the nonparametric (as well as discrete) case on the one hand and the linear case on the other. In the nonparametric case, the constraints involved are becoming better and better understood \citep{TianPearl2002_TestableImplications,ShpitserERR2014_NestedMarkovIntroduction}. In the linear case, specifically for linear structural equation models (LSEMs), these same constraints apply \citep{ShpitserEvansRichardson2018_LSEMNestedMarkov}, as well as others such as the vanishing tetrad constraint \citep{SpirtesGlymourScheines2000} and its higher-dimensional generalizations \citep{SullivantTalaskaDraisma2010}.

In principle, the constraints imposed by a graph in the linear case can be found using methods from computer algebra \citep{Drton2016ARXIV_LSEM_AlgebraicProblems}, yet these can be extremely slow in practice. Further, the constraints are output in the form of long polynomials, which do not immediately lend themselves to a better understanding of these graphical models. In this paper, we present a new, graphical description of such constraints, which is much more succinct and which we believe will help further the understanding of linear structural equation models. In particular, knowledge about constraints helps checking whether different graphs encode observationally indistinguishable hypotheses because the graphs are equivalent in terms of defining the same statistical model. Moreover, constraints may be deployed for goodness-of-fit tests (e.g., \citet{LeungDrton2018}), %
and they can also lead to new criteria to decide parameter identifiability \citep{WeihsRDKKMNRobevaDrton2018_DetGeneralizationsOfIVs}.

This paper is structured as follows. Section~\ref{sec:prelim} discusses preliminaries about linear structural equation models. In Section~\ref{sec:main}, we state our contributions: the definition of graphical constraints, an algorithm for deriving a list of constraints for a given model, and a number of ancillary results. Empirical results are given in Section~\ref{sec:empirical}, and Section~\ref{sec:conclusion} concludes.

\section{Preliminaries}\label{sec:prelim}

In a linear structural equation model (LSEM), the value of each variable $X_v$ is governed by an equation of the form (see e.g.~\citet{FoygelDraismaDrton2012_htc} for a more extensive exposition)
\begin{equation*}
  X_v = \lambda_{0,v} + \sum_w \lambda_{w,v} X_w + \epsilon_v.
\end{equation*}
The structure of an LSEM can be represented by a \emph{mixed graph} $G = (V, D, B)$, where the nodes $V$ index the random variables $X_v$ of the LSEM, each directed edge $v \to w \in D$ denotes that $X_v$ appears on the right-hand side of $X_w$'s structural equation, and $v \leftrightarrow w \in B$ denotes that the noise terms $\epsilon_v$ and $\epsilon_w$ may have nonzero correlation. Note that LSEMs may contain directed cycles, though self-loops ($v \to v$) are not allowed. Also allowed are \emph{bows}: the co-occurrence of a directed and a bidirected edge between the same pair of nodes.

Ignoring the (to us) irrelevant additive terms $\lambda_{0,v}$, the parameters of an LSEM are the matrices $\Lambda$ and $\Omega$. $\Lambda$ must have $\Lambda_{v,w} = 0$ wherever $v \to w \notin D$, and must be such that $I - \Lambda$ is invertible. This last condition is automatically satisfied if $G$ contains no directed cycles. We write $\R^D_{\text{reg}}$ for the space of such $\Lambda$'s. $\Omega$ must have $\Omega_{v,w} = 0$ except on the diagonal and when $v \leftrightarrow w \in B$, and further it must be positive definite. We write $\PD(B)$ for the space of such $\Omega$'s, and write $\PD_n$ for the space of $n \times n$ positive definite matrices.

After observing samples from an LSEM, we can estimate the covariance matrix $\Sigma$ on the observed variables $(X_v)_{v \in V}$. Two central problems ask what we can do when we know $\Sigma$: the \emph{(parameter) identification} problem is to reconstruct the parameters $\Lambda$ and $\Omega$ from $\Sigma$ when the graph is known, and the \emph{structure learning} problem is to find an appropriate graph when only $\Sigma$ is known.

For structure learning, it is important to know for each graph $G$ what set of $\Sigma$'s could be observed for that graph and any choice of parameters. This set, the \emph{model} of $G$, is denoted $\cM(G)$. In general, $\cM(G)$ is a semialgebraic set (we refer to \citet{CoxLittleOShea2015} for further background on algebraic geometry), described by polynomial equalities and inequalities. We will only study the equalities ($f(\Sigma) = 0$) in this paper; as argued by \citet{VanOmmenMooij2017_AlgebraicEquivalence}, the inequalities are of secondary importance for structure learning. The set of polynomials that are zero everywhere on $\cM(G)$ form an \emph{ideal} $I_G$. This ideal can be described by a finite number of \emph{generators} (our constraints play this role).
The \emph{algebraic model} $\overline{\cM(G)}$ is defined as $V(I_G)$, where $V(I)$ denotes the set of points on which all polynomials in ideal $I$ evaluate to zero. In other words, $\overline{\cM(G)} \supseteq \cM(G)$ ignores any inequality constraints imposed by $\cM(G)$.

Each of the constraints describing an algebraic model is an \emph{homogeneous} polynomial: each term has the same total degree. In fact, these polynomials obey a stronger property we call \emph{$V$-homogeneity}: for each model variable $v \in V$, the number of occurrences of $v$ in the indices of the $\sigma$'s (counting $\sigma_{vv}$ twice) is the same for all terms in the polynomial.  Rescaling an LSEM variable by a constant does not affect whether $f(\Sigma) = 0$ iff $f$ is $V$-homogeneous.

\subsection{Rational and Polynomial Constraints}\label{sec:uai_thm}

In order to derive polynomial constraints, we will need expressions for the parameters as a function of $\Sigma$. This is the problem of parameter identification for LSEMs, and a powerful method for this problem is the half-trek criterion (HTC: \citealp{FoygelDraismaDrton2012_htc}). A \emph{half-trek} is a directed path, except that the first edge on the path may instead be bidirected. The set of nodes reachable from $v$ by a half-trek is denoted $\htr(v)$. The HTC-identifiability theorem states that if for each $v \in V$, a set $Y_v$ exists that satisfies certain conditions, then rational functions exists that identify the parameters. We write $\cY$ for the sequence $(Y_v)_v$ of HTC-identifying sets, and $\Lambda_\cY(\Sigma)$ for the function mapping $\Sigma$ to $\Lambda$.

We build on Theorem~1 of \cite{VanOmmenMooij2017_AlgebraicEquivalence}, which shows that for an HTC-identifiable graph $G$ with HTC-identifying sets $\cY = (Y_v)_v$:
\begin{compactenum}
\item For generic $\Sigma \in \cM(G)$, $\Lambda_\cY(\Sigma)$ is defined, is in $\R^D_{\text{reg}}$, and satisfies the \emph{rational constraints}
  \begin{equation}\label{eq:rational}
    [(I - \Lambda_\cY(\Sigma))\trans \Sigma (I - \Lambda_\cY(\Sigma))]_{v,w} = 0%
  \end{equation}
  for all $v \neq w$ with $v \leftrightarrow w \notin B$, $v \notin Y_w$, and $w \notin Y_v$;
\item All $\Sigma \in \cM(G)$ satisfy the \emph{polynomial constraints} obtained by multiplying out the denominators from \eqref{rational};
\item Among $\Sigma$ that satisfy the rational constraints \eqref{rational} and have $\Lambda_\cY(\Sigma) \in \R^D_{\text{reg}}$, generically $\Sigma \in \cM(G)$.
\end{compactenum}
Statements 1 and 2 (showing that $\Sigma \in \cM(G)$ implies satisfaction of constraints) are strong: statement~2, about the polynomial constraints, is not even limited to generic $\Sigma$. Statement~3 in the other direction is weaker. In particular, it does not imply anything for $\Sigma$ where the rational constraints are undefined, or for which $\Lambda_\cY(\Sigma) \notin \R^D_{\text{reg}}$. We will revisit this issue in Section~\ref{sec:ideals}.

\section{Main Results}\label{sec:main}

In this section, we present the definition of graphical constraints, followed by an algorithm that finds such constraints for any HTC-identifiable mixed graph. Section~\ref{sec:ideals} discusses the question of how well a list of graphical constraints describes an algebraic model, and Section~\ref{sec:transformation} introduces transformation operations that may be performed on graphical constraints.

\subsection{Graphical Constraints}

A \emph{graphical constraint} is an undirected bipartite graph (usually a tree), in which each node is labelled with a subset of the model variables $V$.

Each graphical constraint represents a polynomial. This polynomial is the symbolic determinant of the matrix $M$ defined as follows. Let $A$, $B$ be the two parts of the bipartite graphical constraint. The matrix $M$ has a row for each pair $(a,v)$ with $a \in A$ and $v$ in $a$'s label, and a column for each pair $(b,w)$ with $b \in B$ and $w$ in $b$'s label. We use these pairs to index the rows and columns of $M$. We require that $M$ is square. Then $M$ is given by
\begin{equation}\label{eq:matrix}
  M_{(a,v),(b,w)} = \begin{cases}
    \sigma_{v,w} & \text{if $a$ and $b$ are adjacent;}\\
    0 & \text{otherwise.}
  \end{cases}
\end{equation}
We say an observational covariance matrix $\Sigma$ \emph{satisfies} a graphical constraint if the determinant of $M$ evaluates to 0 at $\Sigma$.

A different ordering of the rows and columns of $M$ may flip the sign of the determinant. Because a graphical constraint does not fix a particular ordering, it only represents a polynomial up to sign. This does not pose a problem, since it has no effect on whether or not for given $\Sigma$, the polynomial equals 0.

\begin{figure}[t]
  \centering
  \stackunder{\begin{tikzpicture}[scale=.8]
    \node (v0) at (0,1) {ac};
    \node (v3) at (1,0) {bc};
    \draw (v0) to (v3);
  \end{tikzpicture}}{(a)}
  \qquad
  \stackunder{\raisebox{.5cm}{$\begin{bmatrix}
    \sigma_{ab} & \sigma_{ac}\\
    \sigma_{cb} & \sigma_{cc}
  \end{bmatrix}$}
  }{(b)}
  \qquad
  \stackunder{\begin{tikzpicture}
    \node [circle,fill=black,inner sep=1pt] (a) at (0.0,1.0) [label=180:$\mathstrut a$] {};
    \node [circle,fill=black,inner sep=1pt] (b) at (1.0,1.0) [label=0:$\mathstrut b$] {};
    \node [circle,fill=black,inner sep=1pt] (c) at (1.0,0.0) [label=0:$\mathstrut c$] {};
    \node [circle,fill=black,inner sep=1pt] (d) at (0.0,0.0) [label=180:$\mathstrut d$] {};
    \draw [blue,arrows=
      {_-Stealth[sep,length=1ex]}]
      (a) -- (b);
    \draw [blue,arrows=
      {_-Stealth[sep,length=1ex]}]
      (d) -- (a);
    \draw [blue,arrows=
      {_-Stealth[sep,length=1ex]}]
      (a) to[bend left=15] (c);
    \draw [blue,arrows=
      {_-Stealth[sep,length=1ex]}]
      (c) to[bend left=15] (a);
    \draw [red,dashed,arrows=
      {Stealth[sep,length=1ex]-Stealth[sep,length=1ex]}]
      (a) to[bend left] (b);
  \end{tikzpicture}}{(c)}
  \qquad
  \stackunder{\begin{tikzpicture}[scale=.8]
    \node (v0) at (0,1) {d};
    \node (v1) at (2,1) {ab};
    \node (v2) at (1,0) {ac};
    \node (v3) at (3,0) {d};
    \draw (v0) to (v2);
    \draw (v1) to (v2);
    \draw (v1) to (v3);
  \end{tikzpicture}}{(d)}
  \qquad
  \stackunder{\raisebox{.8cm}{$\begin{bmatrix}
    \sigma_{da} & \sigma_{dc} & 0\\
    \sigma_{aa} & \sigma_{ac} & \sigma_{ad}\\
    \sigma_{ba} & \sigma_{bc} & \sigma_{bd}
  \end{bmatrix}$}
  }{(e)}
  \caption{(a)~a graphical constraint (we draw these in a zig-zag way to make the bipartite parts more easily recognizable); (b)~the matrix represented by this graphical constraint; (c)~an LSEM (with a bow and a directed cycle) imposing a single constraint, namely the graphical constraint displayed in~(d), with (e)~the matrix it represents.}\label{fig:first_examples}
\end{figure}
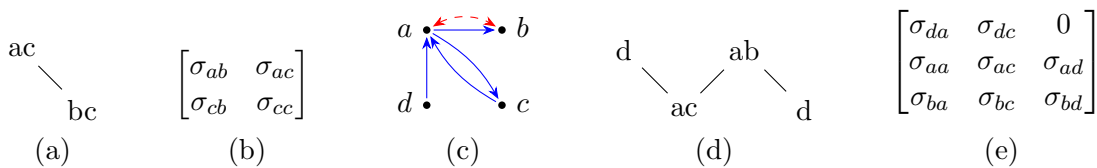
Two examples of graphical constraints are displayed in Figure~\ref{fig:first_examples}~(a) and~(d). The determinant of~(b) is $\sigma_{ab} \sigma_{cc} - \sigma_{ac} \sigma_{bc}$, which is zero iff the partial correlation of $a$ and $b$ given $c$ is zero; this is the Gaussian analogue of conditional independence.

We say a graphical constraint is in normal form if no two nodes have the same set of neighbours. A pair of nodes that have the same set of neighbours can be merged into a single node whose label is the union of the original nodes' labels, without changing the matrix. Restricting graphical constraints to normal form ensures that two different graphical constraints do not describe the same matrix. For example, when the entire tree consists of a central node with some leaf nodes around it, the matrix \eqref{matrix} will have no zero entries, and many similar trees represent the same matrix.

The polynomial represented by a graphical constraint may have exponentially many terms. An immediate advantage of the graphical representation is that evaluating a polynomial for given $\Sigma$ can be made computationally much more efficient by using algorithms for computing determinants (of sparse matrices), compared to evaluating the exponentially many terms. This is useful for testing a constraint on observational data. %

\subsection{Construction Algorithm}\label{sec:htc_construction}

Algorithm~\ref{alg:htc_construction} can be used to find a graphical representation for each of the constraints in \eqref{rational}. The relation between rational/polynomial constraints and the graphical constraints is made precise below in Theorem~\ref{thm:htc_construction}.
\begin{algorithm}[t]
  \SetKwFunction{main}{main}\SetKwFunction{expand}{expand}
  \KwIn{An HTC-identifiable graph $G$, a family $\cY = (Y_v)_v$ of HTC-identifying sets, and a pair $\set{v,w}$ as in \eqref{rational}}
  \KwOut{A graphical constraint}
  \BlankLine
  \nl Start with a graph consisting of two `seed' nodes $t_1$ and $t_2$ joined by an edge, one with label $\set{v}$ and one with label $\set{w}$\;
  \nl \expand{$t_1$}\;
  \nl \expand{$t_2$}\;
  \nl Wherever two or more leaf nodes are attached to the same node, merge them into a single leaf node labelled with the union of those leaf nodes' labels.
  \BlankLine
  \SetKwProg{myproc}{Subroutine}{:}{end}
  \myproc{\expand{$t$}}{
    \nl Let $\set{v}$ be the label of $t$\;
    \nl Replace the label of $t$ by $\set{v} \cup \pa(v)$\;
    \nl For each $y \in Y_v$, create a new node $t_y$ attached to $t$ with label $\set{y}$\;
    \nl For each $y \in Y_v$ with $y \in \htr(v)$, \expand{$t_y$}\;
    }
  \caption{HTC-based graphical constraint construction.}\label{alg:htc_construction}
\end{algorithm}

Line~4 of the algorithm brings the constraint into normal form. It suffices to check leaf nodes because the graph is a tree, so the only pairs of nodes that could have the same set of neighbours are leaf nodes.

Figure~\ref{fig:htc_construction} displays an example where Algorithm~\ref{alg:htc_construction} is applied to the graph~(a). This graph is bow-free and acyclic, in which case it is possible to take $Y_v = \pa(v)$ as the HTC-identifying sets $\cY$. With this input, the algorithm produces the graphical constraint in~(b). This constraint represents the matrix~(c), whose determinant is a polynomial of degree~3 with 4~terms. The LSEM in Figure~\ref{fig:first_examples}~(c) is not bow-free or acyclic; with $Y_a=\set{c,d}$, $Y_b=Y_c=\set{d}$ and $Y_d = \myempty$, the algorithm outputs the graphical constraint in Figure~\ref{fig:first_examples}~(d).

\begin{figure}[t]
  \centering
  \stackunder{\begin{tikzpicture}
    \node [circle,fill=black,inner sep=1pt] (a) at (0.866025403784,1.0) [label=90.0:$\mathstrut a$] {};
    \node [circle,fill=black,inner sep=1pt] (b) at (0.866025403784,0.0) [label=270.0:$\mathstrut b$] {};
    \node [circle,fill=black,inner sep=1pt] (c) at (0.0,0.5) [label=180.0:$\mathstrut c$] {};
    \node [circle,fill=black,inner sep=1pt] (d) at (1.73205080757,0.5) [label=0.0:$\mathstrut d$] {};
    \draw [blue,arrows=
      {_-Stealth[sep,length=1ex]}]
      (a) -- (b);
    \draw [red,dashed,arrows=
      {Stealth[sep,length=1ex]-Stealth[sep,length=1ex]}]
      (a) -- (c);
    \draw [red,dashed,arrows=
      {Stealth[sep,length=1ex]-Stealth[sep,length=1ex]}]
      (a) -- (d);
    \draw [red,dashed,arrows=
      {Stealth[sep,length=1ex]-Stealth[sep,length=1ex]}]
      (b) -- (c);
    \draw [blue,arrows=
      {_-Stealth[sep,length=1ex]}]
      (b) -- (d);
  \end{tikzpicture}}{(a)}
  \qquad
  \stackunder{\begin{tikzpicture}[scale=.8]
    \node (v0) at (0,1) {\textcircled{c}};
    \node (v3) at (1,0) {b\textcircled{d}};
    \node (v1) at (2,1) {ab};
    \node (v4) at (3,0) {a};
    \draw (v0) to (v3);
    \draw (v1) to (v3);
    \draw (v1) to (v4);
  \end{tikzpicture}}{(b)}
  \qquad
  \stackunder{\raisebox{.8cm}{$\begin{bmatrix}
    \sigma_{cb} & \sigma_{cd} & 0\\
    \sigma_{ab} & \sigma_{ad} & \sigma_{aa}\\
    \sigma_{bb} & \sigma_{bd} & \sigma_{ba}
  \end{bmatrix}$}
  }{(c)}
  \caption{(a)~an LSEM; %
    (b)~the graphical constraint output by Algorithm~\ref{alg:htc_construction} (with the original seed node labels marked by circles); (c)~the matrix represented by this graphical constraint.}\label{fig:htc_construction}
\end{figure}
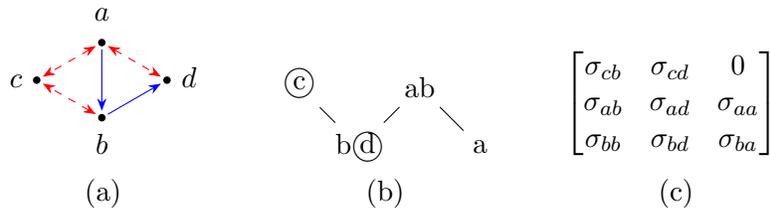

\begin{theorem}\label{thm:htc_construction}
  The polynomial represented by the graphical constraint from Algorithm~\ref{alg:htc_construction} is either equal to or a multiple of the numerator of the rational constraint \eqref{rational}.
\end{theorem}

\begin{proof}
  At each point in the algorithm \emph{after} an expansion step, we have a tree graph in which some nodes might still have to be expanded. Take the seed node that is expanded first to be the root of each of these trees. To each of these intermediate trees, we associate a matrix like \eqref{matrix}, except that at non-zero positions $y,v$ or $v,y$ where $y$ is to-be-expanded, the matrix entry is not $\Sigma_{y,v}$ but $[(I-\Lambda_\cY(\Sigma))\trans \Sigma]_{y,v}$. Here and in the following, we make slight abuse of notation and take $v$, $\pa(v)$, $Y_v$, etc.~to refer both to labels (i.e.~elements of $V$) and to the specific row or column indices corresponding to the (node, label element)-pair under consideration. %

  The proof proceeds by induction. Our induction hypothesis is that after each expansion step in the algorithm, the matrix at that point has a determinant that is a multiple of \eqref{rational}.

  First we prove that for any vector $\bc = [c_v, c_{p_1}, \ldots, c_{p_n}]\trans$ where $\set{p_1, \ldots, p_n} = \pa(v)$,
  \begin{equation}\label{eq:lemma}
    \lvert \bA^{(v)} \rvert \cdot [(I-\Lambda_\cY(\Sigma))_{\set{v} \cup \pa(v); v}]\trans \bc
    = \begin{vmatrix}
      c_v & c_{p_1} \cdots c_{p_n}\\
      \bb^{(v)} & \bA^{(v)}
      \end{vmatrix}.
  \end{equation}
  Here $\bA^{(v)}$ and $\bb^{(v)}$ are the matrix and vector from \cite[proof of Theorem~1]{FoygelDraismaDrton2012_htc}, where they are used to identify $\Lambda_\cY(\Sigma)_{\pa(v),v}$ as the solution of $\bA^{(v)} \cdot \Lambda_\cY(\Sigma)_{\pa(v),v} = \bb^{(v)}$. They are defined by $\bA^{(v)} = \Sigma_{Y_v, \pa(v)}$ and $\bb^{(v)} = \Sigma_{Y_v,v}$, except in rows where $y \in \htr(v)$: there $\bA^{(v)}_{y, \pa(v)} = [(I-\Lambda_\cY(\Sigma))\trans \Sigma]_{y,\pa(v)}$ and $\bb^{(v)}_{y} = [(I-\Lambda_\cY(\Sigma))\trans \Sigma]_{y,v}$.
  For each $p \in \pa(v)$, by Cramer's rule, $\Lambda_\cY(\Sigma)_{p,v} = \lvert \bA^{(v)}_p \rvert / \lvert \bA^{(v)} \rvert$, where $\bA^{(v)}_p$ is the matrix obtained from $\bA^{(v)}$ by replacing column $p$ by $\bb^{(v)}$. So
  \begin{align*}
    \lvert \bA^{(v)} \rvert \cdot [ I - \Lambda_\cY(\Sigma) ]_{\set{v} \cup \pa(v); v}
    &= \left[ \lvert\bA^{(v)}\rvert, -\lvert\bA^{(v)}_{p_1}\rvert, \ldots, -\lvert\bA^{(v)}_{p_n}\rvert \right]\trans.
  \end{align*}
  For $\bc$ a standard basis vector, \eqref{lemma} can be verified by noting that for $\bc$ with $c_p=1$ ($p \in \pa(v)$),
  \begin{equation*}
    -\lvert \bA^{(v)}_p \rvert
    =
    -\begin{vmatrix}
      1 & \mathbf{0}\\
      \mathbf{0} & \bA^{(v)}_p
    \end{vmatrix}
    =
    \begin{vmatrix}
      0 & c_{p_1} \cdots c_{p_n}\\
      \bb^{(v)} & \bA^{(v)}
    \end{vmatrix},
  \end{equation*}
  where the sign flips due to the transposition of two columns. Now the case of general $\bc$ follows by multilinearity of the determinant.

  \emph{(base case)} The first expansion performed by the algorithm turns a tree with only the two seed nodes $v$ and $w$ into a tree where only root node $v$ has been expanded, so all other nodes in the tree are children of $v$. The corresponding matrix has columns $\set{v} \cup \pa(v)$ and rows $\set{w} \cup Y_v$ (note that the roles of rows and columns can be interchanged). Because $w$ is to-be-expanded, entry $w,x$ in $w$'s row equals $[\Sigma (I-\Lambda_\cY(\Sigma))]_{x,w}$. %
  For each $y \in Y_v$, that $y$ is to-be-expanded iff $y \in \htr(v)$, so the other rows of the matrix equal either entries of $\Sigma$ or of $(I-\Lambda_\cY(\Sigma))\trans \Sigma$, the choice agreeing with the choice in the definitions of $\bA^{(v)}$ and $\bb^{(v)}$, establishing that this matrix equals the one in the right-hand side of \eqref{lemma} for $c_x = [\Sigma (I-\Lambda_\cY(\Sigma))]_{x,w}$.
  By \eqref{lemma}, the determinant equals a multiple of
  \begin{equation*}
    [(I - \Lambda_\cY(\Sigma))\trans \Sigma (I - \Lambda_\cY(\Sigma))]_{v,w},
  \end{equation*}
  proving that the induction hypothesis holds after the first expansion.

  \emph{(induction step)} For each subsequent expansion, let $v$ be the node that is expanded in this step, so that its label after expansion is $\set{v} \cup \pa(v)$ and the union of its children's labels is $Y_v$. Write $L$ for the label of $v$'s parent in the tree (it is the same before and after the expansion). The matrix $M'$ before expansion differs from the matrix $M$ after expansion in that the rows and columns corresponding to $\pa(v)$ and $Y_v$ are absent in $M'$, and the entries $v,\ell$ (or $\ell, v$) for $\ell \in L$ are equal to $[(I-\Lambda_\cY(\Sigma))\trans \Sigma]_{v, \ell}$. We need to show that the determinant of the matrix $M$ is a multiple of the determinant of the matrix $M'$.

  Because $\lvert \set{v} \cup \pa(v) \rvert = \lvert Y_v \rvert + 1$, each term in the expansion of the determinant of $M$ contains exactly one entry from $(\set{v} \cup \pa(v)) \times L$. So
  \begin{align*}
    \lvert M \rvert &= \sum_{\ell \in L} \lvert M_{\text{with $(\set{v} \cup \pa(v)) \times (L \setminus \set{\ell})$ replaced by zeros}} \rvert\\
    &= \sum_{\ell \in L} \lvert M_{(\set{v} \cup \pa(v)); (\set{\ell} \cup Y_v)} \rvert \cdot \lvert M_{\text{other rows and columns}} \rvert\\
    &= \lvert \bA^{(v)} \rvert \sum_{\ell \in L} [(I-\Lambda_\cY(\Sigma))_{\set{v} \cup \pa(v); v}]\trans \Sigma_{\set{v} \cup \pa(v); \ell} \cdot \lvert M_{\text{other rows and columns}} \rvert\\
    &= \lvert \bA^{(v)} \rvert \cdot \lvert M' \rvert.
  \end{align*}
  This completes the proof.
\end{proof}

This theorem shows that whenever the polynomial constraint from Section~\ref{sec:uai_thm}, point~2 is satisfied, then so is the corresponding graphical constraint.

The proof also makes explicit by what factors the rational constraint is multiplied: for each expansion of a node with label $\set{v}$, there is a multiplication by $\lvert \bA^{(v)} \rvert$. As the final product is a polynomial, clearly the product of these factors is a multiple of the rational constraint's denominator. So if the rational constraints are undefined for some $\Sigma$, this must mean that a relevant $\lvert \bA^{(v)} \rvert = 0$.
By inspecting the construction, we can also isolate from the graphical constraint the part that represents the polynomial $\lvert \bA^{(v)} \rvert$: it is the graphical constraint obtained by starting with a single node with label $\set{v}$ and expanding it, then finally removing the label $v$ from the starting node. %

\subsection{Describing Algebraic Models Using Graphical Constraints}\label{sec:ideals}

In terms of algebraic geometry (for which we refer to \citet{CoxLittleOShea2015}), the algebraic model $\overline{\cM(G)}$ is an irreducible variety, which is described by a prime ideal. The constraints found by Algorithm~\ref{alg:htc_construction} generate an ideal that is not necessarily prime. There are two ways in which an ideal can fail to be prime, using that an ideal is prime iff it is \emph{radical} and \emph{primary}:
\begin{compactitem}
\item An ideal $I$ is nonradical if $f^k \in I$ for some polynomial $f$ and some $k \geq 2$, but $f \notin I$. The radical $\sqrt{I}$ of the ideal $I$ also includes all such polynomials $f$. We are mostly concerned with the ideal's ability to describe a model, and $V(I) = V(\sqrt{I})$, so failure to be radical does not impact our goal.
\item If an ideal $I$ is nonprimary, then $V(I)$ is the union of multiple irreducible varieties. For instance, the ideal generated by $xy$ describes a variety that is the union of the two hyperplanes $x=0$ and $y=0$. Primary ideals describing these irreducible varieties can be found computationally using \emph{primary decomposition} algorithms, but these algorithms can be extremely computationally expensive. It follows from Theorem~1 of \cite{VanOmmenMooij2017_AlgebraicEquivalence} that one of the component varieties is the algebraic model we want to describe, meaning that the remaining components are spurious.
\end{compactitem}
The spurious components of a nonprimary ideal lead to false positives: sets of $\Sigma$ that satisfy the polynomial constraints but do not belong to the algebraic model. We may further classify ideals by the types of spurious primary components to which they give rise:
\begin{compactitem}
\item If an ideal $I_{\text{comp}}$ in the primary decomposition
  contains (a power of) a principal minor of $\Sigma$, then
  $V(I_{\text{comp}})$ only consists of $\Sigma$ for which this
  principal minor evaluates to 0. Such $\Sigma$ are on the boundary of
  the positive definite cone, so $V(I_{\text{comp}}) \cap \PD_n =
  \myempty$ and this spurious component does not impact the
  constraints' ability to distinguish which covariance matrices are in
  the algebraic model. If all spurious components of an ideal have
  this property, we call the ideal \emph{PD-primary}.
\item Suppose some $f \in I_{\text{comp}}$ contains a monomial
  $\prod_v \sigma_{vv}^{k_v}$ for some sequence $(k_v)_v$; i.e., a monomial consisting entirely of elements from the
  diagonal of $\Sigma$. If $f$ is $V$-homogeneous, there can be at
  most one monomial of this form. Then the identity matrix $I$, or in fact any diagonal $\Sigma \in \PD_n$,
  is not in $V(I_{\text{comp}})$, because $f(\Sigma) = \prod_v
  \sigma_{vv}^{k_v} > 0$. If all spurious components of an ideal
  contain such a monomial, we call the ideal \emph{$I$-primary}.

  All algebraic models we consider in this
  paper \emph{do} contain all such diagonal $\Sigma$ --- even the
  smallest model, described by the empty mixed graph, contains all
  these. This makes it clear that this component is spurious. So in
  this case, it may be that $V(I_{\text{comp}}) \cap \PD_n \neq
  \myempty$, but we do have that the polynomial constraints contain
  all the information needed to uniquely indicate an algebraic
  model.
\item In more general cases, the spurious component may be equal to
  an algebraic model other than the one we mean to describe. We may
  be able to distinguish which component is correct and which are
  spurious by referencing the denominators in the rational constraints
  \eqref{rational}. But the polynomial constraints by themselves may not
  contain enough information to indicate a single algebraic model.
\end{compactitem}
Note that primary ideals are PD-primary, and PD-primary ideals are $I$-primary.

As mentioned in Section~\ref{sec:uai_thm}, Theorem~1 of \cite{VanOmmenMooij2017_AlgebraicEquivalence} leaves open the possibility of $\Sigma$ that satisfy the polynomial constraints but are not in $\cM(G)$, when either $\Lambda_\cY$ is undefined at $\Sigma$ or $\Lambda_\cY(\Sigma) \notin \R^D_{\text{reg}}$. This final case can only occur if $G$ contains a directed cycle, and we will not address it further in this paper. The case of  $\Lambda_\cY$ being undefined arises when a matrix $\bA^{(v)}$ in the HTC identification procedure is singular (see the end of Section~\ref{sec:htc_construction}). We can use this to show that certain ideals are well-behaved:

\begin{proposition}
  If $G$ is ancestral, the ideal generated by the graphical constraints of Algorithm~\ref{alg:htc_construction} with the choice $Y_v = \pa(v)$ for each $v$ is PD-primary.
\end{proposition}
\begin{proof}
  In an ancestral graph, $\pa(v) \cap \htr(v) = \myempty$, so the \texttt{expand} subroutine will not call itself recursively. Then the graphical constraint representing $\lvert \bA^{(v)} \rvert$ (see the end of Section~\ref{sec:htc_construction}) has two nodes, both with label $\pa(v)$. This describes a principal minor, implying that the ideal in question is PD-primary.
\end{proof}
\begin{proposition}\label{prop:bap}
  If $G$ is bow-free and acyclic, the ideal generated by the graphical constraints of Algorithm~\ref{alg:htc_construction} with the choice $Y_v = \pa(v)$ for each $v$ is $I$-primary.
\end{proposition}
\begin{proof}
  Since $G$ is acyclic, any spurious components must be contained in regions where the rational constraints are not defined, meaning they must each lie within one of the varieties defined by $\lvert \bA^{(v)} \rvert = 0$ for some $v$. For the choice $Y_v = \pa(v)$, we see from the graphical constraint representing $\lvert \bA^{(v)} \rvert$ that the expansion of $\lvert \bA^{(v)} \rvert$ contains a unique monomial $\prod_v \sigma_{vv}^{k_v}$. This implies that the ideal in question is $I$-primary.
\end{proof}
We will see in Section~\ref{sec:empirical} that these sufficient conditions are far from necessary.

\subsection{Transformation of Graphical Constraints}\label{sec:transformation}

Performing elementary row and column operations on the matrix \eqref{matrix} does not change its determinant except by a constant factor. In particular, the operation of adding or subtracting one row or column from another will yield a matrix that can again be represented by a graphical constraint, if the two involved rows/columns were labelled with the same LSEM variable and the entries of the resulting matrix are still all either $0$ or $\sigma_{vw}$. Then these graphical constraints are equivalent, even though they are not isomorphic as graphs.

An application of these matrix transformations is to remove spurious factors. Theorem~\ref{thm:htc_construction} shows that the polynomial represented by a graphical constraint can be equal to the numerator of the corresponding rational constraint, or it can contain additional factors. Getting rid of such factors can sometimes remove problematic primary components from the ideal. If a matrix can be transformed into a $2 \times 2$ block diagonal matrix, this shows that its determinant factorizes into the determinants of the two diagonal blocks. Below we illustrate a class of such steps that can be recognized graphically.

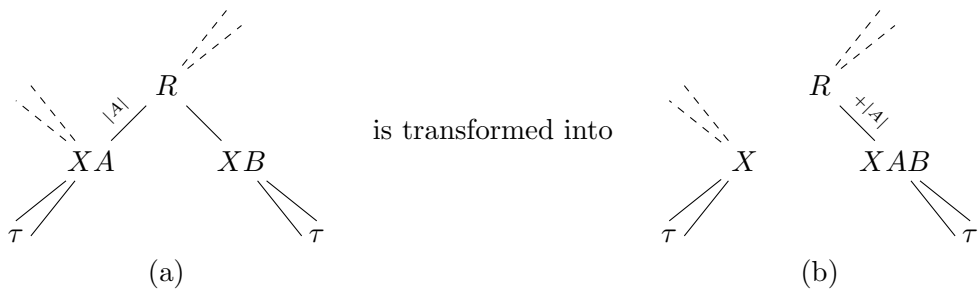
\begin{figure}[t]
  \centering
  \stackunder{\begin{tikzpicture}%
    \node (vr) at (1,1) {$R$};
    \node (va) at (0,0) {$X A$};
    \node (vb) at (2,0) {$X B$};
    \node (vta) at (-1,-1) {$\tau$};
    \node (vtb) at (3,-1) {$\tau$};
    \draw (vr) to  node [midway,sloped,above] {\tiny $\lvert A \rvert$} (va);
    \draw (vr) to (vb);
    \draw (va) to (-0.8,-1.0);
    \draw (va) to (-1.0,-0.8);
    \draw (vb) to (2.8,-1.0);
    \draw (vb) to (3.0,-0.8);
    \draw[dashed] (va) to (-0.8,1.0);
    \draw[dashed] (va) to (-1.0,0.8);
    \draw[dashed] (vr) to (1.8,2.0);
    \draw[dashed] (vr) to (2.0,1.8);
  \end{tikzpicture}}{(a)}
  \quad \raisebox{1.5cm}{is transformed into} \quad
  \stackunder{\begin{tikzpicture}%
    \node (vr) at (1,1) {$R$};
    \node (va) at (0,0) {$X$};
    \node (vb) at (2,0) {$X A B$};
    \node (vta) at (-1,-1) {$\tau$};
    \node (vtb) at (3,-1) {$\tau$};
    \draw (vr) to  node [midway,sloped,above] {\tiny $+\lvert A \rvert$} (vb);
    \draw (va) to (-0.8,-1.0);
    \draw (va) to (-1.0,-0.8);
    \draw (vb) to (2.8,-1.0);
    \draw (vb) to (3.0,-0.8);
    \draw[dashed] (va) to (-0.8,1.0);
    \draw[dashed] (va) to (-1.0,0.8);
    \draw[dashed] (vr) to (1.8,2.0);
    \draw[dashed] (vr) to (2.0,1.8);
  \end{tikzpicture}}{(b)}
  \caption{(a)~Transformation preconditions: all trees $\tau$ at the right node have copies at the left node and the weight of the left edge equals $\lvert A \rvert$;
    (b)~result of the transformation.}\label{fig:transformation}
\end{figure}
This transformation applies to tree-shaped constraints and revolves around three nodes (called left, right, and central here). Let $X$ be the intersection of left's and right's labels, and $A$ and $B$ the labels in left but not in right and vice versa. The transformation's preconditions are illustrated in Figure~\ref{fig:transformation}~(a): For each subtree at the right node (not containing the central node), there is an identical subtree at the left node; and the weight of the edge between the central and left nodes equals $\lvert A \rvert$. Here the \emph{weight} of an edge counts the number of occurrences in a term of the determinant of the elements from the submatrix of $M$ with rows and columns corresponding to the endpoints of this edge; for a tree-shaped graphical constraint, this number is the same for all terms in the determinant.

The result of the transformation is displayed in Figure~\ref{fig:transformation}~(b): the edge between the central and left nodes is deleted, and the nonshared labels of the left node are removed there and added to the right node.

\begin{proposition}
  The transformation described above preserves the represented polynomial up to sign.
\end{proposition}
\begin{proof}
Assume w.l.o.g.~that the left and right nodes are columns. Add `cross-edges' between the left node and the roots of the right $\tau$-subforest by adding to each row in the right $\tau$-subforest the corresponding row from the left $\tau$-subforest, and similarly reducing columns in the left $\tau$-subforest by their right counterparts. Then reduce the columns of the left node's shared labels $X$ by their right counterpart. Putting the left node with all attached parts first in the matrix and the central and right nodes with attached parts second, we find that the diagonal blocks are square and that all entries in the bottom left block are 0, so making the top right block 0 as well does not change the determinant.
\end{proof}

\section{Computational Study}\label{sec:empirical}

Proposition~\ref{prop:bap} shows that for bow-free acyclic graphs, with the choice $Y_v = \pa(v)$, Algorithm~\ref{alg:htc_construction} is guaranteed to produce an $I$-primary ideal. Since PD-primary ideals are preferable, we ran a computational study over all bow-free acyclic graphs on five variables with nine edges. (For these graphs, the model is described by exactly one constraint, and the type of ideal generated by this constraint can be determined by its factorization. Graphs with fewer edges impose more constraints, the analysis of which was not computationally feasible.)

The results were as follows: Without transformation, among the 86 algebraic equivalence classes (counting isomorphic classes only once, so e.g.~the class imposing graphical constraint $a$---$b$ and the one imposing $d$---$e$ are together counted once), there were 5 classes containing mixed graphs for which the algorithm produced a non-PD-primary ideal (each of these classes also contained graphs producing PD-primary ideals).
Using the transformation displayed in Figure~\ref{fig:transformation}, all non-PD-primary ideals were converted to PD-primary ones.
So in practice, the ideals are often of a higher quality than could be guaranteed by Proposition~\ref{prop:bap}.

The above is a measure of the utility of Algorithm~\ref{alg:htc_construction} and of the transformation from Figure~\ref{fig:transformation}. We also want to know whether the class of graphical constraints has enough expressive power to represent arbitrary polynomial constraints, preferably without spurious factors. To study this, for each algebraic equivalence class on five nodes with at least nine edges per mixed graph, allowing bows but not directed cycles, we tried the algorithm on all mixed graphs in the class and all HTC-identifying choices of $\cY$. If we did not find a graphical constraint without spurious factors this way, we also tried various transformations or even brute-force search. We did the same for four-node graphs with at least five edges allowing bows and directed cycles.

\begin{table}
\centering
\begin{tabular}{llrr}
\toprule
primary & PD-primary & 5-node acyclic & 4-node general \\
\midrule
tree & tree & 374 & 16\\
nontree(?) & tree & 2 & 0\\
none & tree & 2 & 0\\
? & tree & 3 & 0\\
? & ? & 17 & 3\\
\midrule
\multicolumn{2}{l}{total} & 398 & 19\\
\bottomrule
\end{tabular}
\caption{Number (up to isomorphism) of algebraic equivalence classes imposing one constraint for which certain types of graphical constraints are known to exist.}\label{tab:existence}
\end{table}

The results are summarized in Table~\ref{tab:existence}.
Among the five-node graphs, there were 18 equivalence classes without any HTC-identifiable members, so Algorithm~\ref{alg:htc_construction} could not be used. One of these classes imposed a polynomial constraint of degree 7. For this constraint we found a graphical form via brute-force search.
The other 17 equivalence classes impose constraints of degrees 9 through 14, for which such a brute-force search was infeasible.

It is noteworthy that in all classes where we know at least one graphical constraint, we know a (tree-shaped) PD-primary constraint. In the majority of classes, this constraint is in fact primary. In two cases, all primary graphical constraints are nontrees (confirmed by brute-force search in one of these two cases), and in two other cases, a brute-force search found no primary graphical constraints of any kind. In three further cases, primary graphical constraints have not been found but might exist.

\section{Conclusion}\label{sec:conclusion}

We have presented a graphical formalism to give succinct representations of polynomials that appear as constraints on the algebraic model of an LSEM, and illustrated the expressive power of these graphical constraints for purposes of describing algebraic models.

We also presented an algorithm that derives a list of graphical constraints for any HTC-identifiable graph, and a transformation that can be used to simplify some graphical constraints. While our empirical results show they are quite powerful at their tasks, there is room for improvement here. A number of generalizations of HTC have been proposed, such as EID and TSID \citep{WeihsRDKKMNRobevaDrton2018_DetGeneralizationsOfIVs}, and generalizing our algorithm to work with these would enable the derivation of graphical constraints for a larger class of mixed graphs. Improved algorithms might also avoid some of the spurious factors which are sometimes produced by our current algorithm.

Another important direction for future work is to define a graphical separation criterion along the lines of d-separation and t-separation \citep{SullivantTalaskaDraisma2010}, which, given a mixed graph and a graphical constraint, allows us to decide whether that constraint holds in the graph's model. Such a criterion may build on the restricted trek separation criterion of \citet{DrtonRobevaWeihs2020_RestrictedTrekSeparation}: our graphical constraints can be seen as recursively nested determinants that obey additional structural restrictions.

\appendix
\acks{This project has received funding from the European Research Council (ERC) under the European Union's Horizon 2020 research and innovation programme (grant agreement No 883818).}
\vskip 0.2in
\bibliography{../caus}

\end{document}